\documentclass[12pt]{article}%
\usepackage{amsmath}
\usepackage{amsfonts}
\usepackage{amssymb}
\usepackage{graphicx}%
\providecommand{\U}[1]{\protect\rule{.1in}{.1in}}
\newtheorem{theorem}{Theorem}

\newtheorem{lemma}[theorem]{Lemma}

\newenvironment{proof}[1][Proof]{\noindent\textbf{#1.} }{\ \rule{0.5em}{0.5em}}
\begin{document}

\title{On a question related to a basic convergence theorem of Harish-Chandra}
\author{Nolan R. Wallach}
\maketitle

\begin{abstract}
In his first 1958 paper on zonal spherical functions Harish-Chandra proved an
extremely delicate convergence theorem which was basic to his subsequent
definition of his Schwartz space and his theory of cusp forms. This paper
gives elementary proofs that a related integral converges for for groups of
real rank one, several groups of real rank 2 (including $SO(n,2),Sp_{4}%
(\mathbb{R})$  and $Sp_{4}(\mathbb{C})$), $GL(n,\mathbb{R})$ and
$GL(n,\mathbb{C})$. In fact, a stronger result has been proved in
\cite{raphael}. Applications of the question are also studied.

\end{abstract}

\section{Introduction}

Let $G$ be a real reductive group (for example $GL(n,\mathbb{R}$) with maximal
compact subgroup $K$ ($O(n)$) \ and corresponding Cartan involution $\theta$
($g\mapsto(g^{T})^{-1})$ and let $G=\theta(N)AK$ be an Iwasawa decomposition
of $G$ (the Gram-Schmidt process) with $N$ a maximal unipotent subgroup (upper
triangular with ones on the main diagonal) and $A$ is the identity component
of a maximal $\mathbb{R}$--split torus of $G$ normalizing $N$ (the diagonal
matrices with positive entries). Let
\[
\alpha^{\rho}=(\det Ad(a)_{|_{Lie(N)}})^{\frac{1}{2}}%
\]
for $g\in G$, $g=vak,v\in\theta(N),a\in A,k\in K$ set $a(g)=a$. One of
Harish-Chandra's most delicate results (\cite{Spher1},c.f.\cite{RRGI}, Theorem
4.5.4 ) which is critical to his method of cusp forms in his proof of the
Plancherel Theorem for a real reductive group is that there exists $d$ such
that%
\[
\int_{N}a(n)^{-\rho}(1+\log a(n)^{\rho})^{-d}dn<\infty.
\]
To appreciate the delicacy of this result you should try to prove it by hand
in the case of $SL(4,\mathbb{R})$. Here if
\[
X=\left[
\begin{array}
[c]{cccc}%
1 & x_{1} & x_{2} & x_{3}\\
0 & 1 & x_{4} & x_{5}\\
0 & 0 & 1 & x_{6}\\
0 & 0 & 0 & 1
\end{array}
\right]
\]
then the Harish-Chandra theorem, combined with his result that there exist
$c_{1},c_{2}>0$ such that
\[
c_{1}(1+\log(1+\left\Vert x\right\Vert ^{2})\leq1+\log(a(n)^{\rho})\leq
c_{2}(1+\log(1+\left\Vert x\right\Vert ^{2}),
\]
is the same as the assertion that there exists $r$ such that%
\[
\int_{\mathbb{R}^{6}}\frac{dx}{\varphi(x)^{\frac{1}{2}}(1+\log(1+\left\Vert
x\right\Vert ^{2})^{r}}<\infty
\]
with%
\[
\varphi(x)=a(X)^{2\rho}=(1+x_{1}^{2}+x_{2}^{2}+x_{3}^{2})\times
\]%
\[
(1+x_{4}^{2}+x_{5}^{2}+(x_{2}-x_{1}x_{4})^{2}+(x_{3}-x_{1}x_{5})^{2}%
+(x_{3}x_{4}-x_{2}x_{5})^{2})\times
\]%
\[
(1+x_{6}^{2}+(x_{5}-x_{4}x_{6})^{2}+(x_{3}-x_{1}x_{5}-x_{2}x_{6}+x_{1}%
x_{4}x_{6})^{2})
\]
which is not obvious. In fact, except for the power of the log term in the
Harish-Chandra result it is best possible. Indeed, on can easily show that if
$\varphi$ is a real polynomial on $\mathbb{R}^{n}$ of degree $2n$ and
$\varphi(x)\geq1$ then
\[
\int_{\left\Vert x\right\Vert \leq r}\varphi(x)^{-\frac{1}{2}}dx\geq
C(1+\log(1+r^{2})).
\]

In this paper I study a related question.

Is
\[
\int_{\lbrack N,N]}a(n)^{-\rho}(1+\log a(n)^{\rho})^{d}dn<\infty
\]
for all $d>0$?

For $GL(4,\mathbb{R})$, noting that $n\in\lbrack N,N]$ if and only if
$x_{1}=x_{4}=x_{6}=0$ one has
\[
\varphi(0,x_{2},x_{3},0,x_{5},0)=(1+x_{2}^{2}+x_{3}^{2})((1+x_{2}^{2}%
)(1+x_{5}^{2})+x_{3}^{2})(1+x_{3}^{2}+x_{5}^{2}).
\]
So one can show that%
\[
\varphi(0,x_{2},x_{3},0,x_{5},0)\geq(1+x_{2}^{2})^{\frac{7}{6}}(1+x_{3}%
^{2})^{\frac{7}{6}}(1+x_{5}^{2})^{\frac{7}{6}}
\]
and, thus, the answer to our question for $GL(4,\mathbb{R})$ is yes. However,
In the case of $GL(5,\mathbb{R})$ the integral in the question is just as
intractable as the Harish-Chandra integral for $GL(4,\mathbb{R)}$.

An affirmative answer would imply that if $f$ satisfies Harish-Chandra's weak
inequality and if $P$ is a parabolic subgroup of $G$ with unipotent radical
$N_{P}$ (see section \ref{prelim} for the meanings of these terms) then%
\[
\int_{\lbrack N_{P},N_{P}]}\left\vert f(n)\right\vert dn<\infty.
\]
and if $g$ is in the Schwartz space of $[N_{P},N_{P}]\backslash N_{P}$ (which
is isomorphic with $\mathbb{R}^{\dim[N_{P},N_{P}]\backslash N_{P}}$) then%
\[
\int_{\lbrack N_{P},N_{P}]\backslash N_{P}}\int_{[N_{P},N_{P}]}f(nx)dng(x)dx
\]
defines a tempered distribution on $[N_{P},N_{P}]\backslash N_{P}$. This can
be used to give a more direct proof of Harish-Chandra's formula for the
Harish-Chandra transform of wave packets and the relationship between the
Fourier transform of a Harish-Chandra wave packet and a Whittaker wave packet.

The main result of this paper is an elementary proof of an affirmative answer
for $GL(n,F),F=\mathbb{R}$ or $\mathbb{C}$ ( in fact for any group whose
commutator group os locally isomorphic with $SL(n,F)$) . In addition it is
observed that the answer is yes for real rank one groups and for certain
groups of real rank 2 (including $Sp_{4}(F),F=\mathbb{R}$ or $\mathbb{C}$,
$SO(n,2)$ and $SU(2,2)$).

The author has been informed that the answer is affirmative in all cases by
Rapha\"{e}l Beauzart-Plessis who points out that Proposition B.3.1 in
\cite{raphael}\ implies the stronger assertion%
\[
\int_{\lbrack N,N]}a(n)^{-\rho+\varepsilon\rho}dn<\infty
\]
for some $\varepsilon>0$. The methods in this paper are much more elementary
than those in \cite{raphael} so even though they are superseded they still may
be useful. Especially to thos who are interested in a simple proof of the
Harish-Chandra Plancherel and Whittaker Plancherel Therem for $GL(n)$.

We will now describe the organization of this paper. After some preliminaries
(in which implications in this introduction are explained), we first prove
that the answer to the question is yes for groups of real rank 1 and for some
real rank 2 groups. The rest of the paper is devoted to the case when
$G=SL(n,F)$,$F=\mathbb{R}$ or $\mathbb{C}$. For this we first give a (perhaps)
new proof of Harish-Chandra's result which in particular shows that the $d$ in
the \ statement can be taken to be $\dim N+\varepsilon$ for any $\varepsilon
>0$. We then follow the steps in that proof to prove that the integrals in
question are finite by induction on $n$. In order to carry out the induction a
stronger result is proved.

\section{Preliminaries\label{prelim}}

For lack of a term we will call a real reductive group tame if the answer to
our question above is yes. Note that $G$ is tame if and only if every normal,
simple subgroup of $G$ is tame. So in this paper we will assume that $G$ is
simple. Let $K$ be a maximal compact subgroup of $G$ and let $G=NAK$ be an
Iwasawa decomposition of $G$. Let $\Phi(N)$ be the set of weights of $A$ on
$Lie(N)$. Set
\[
A^{+}=\{a\in A|a^{\alpha}\geq1,\alpha\in\Phi(N)\}.
\]
Let $a^{\rho}=\det(Ad(a))$, as in the introduction. Noting that $G=KA^{+}K$,
set%
\[
\left\Vert k_{1}ak_{2}\right\Vert =a^{\rho},k_{1},k_{2}\in K,a\in A^{+}.
\]
We note that this expression depends only on $g=k_{1}ak_{2}$. Indeed, if we
put an $Ad(G)$ invariant symmetric bilinear form on $Lie(G)$, $B$, \ such that
if $\theta$ \ is the Cartan involution corresponding to $K$ then%
\[
\left\langle X,Y\right\rangle =-B(X,\theta Y)
\]
defines an inner product on $Lie(G)$. Then if $n=\dim N$ then $\left\Vert
g\right\Vert ^{2}$ is equal to the operator norm of $\wedge^{n}Ad(g)$ relative
to the norm corresponding to $\left\langle ...,...\right\rangle $ on
$\wedge^{n}Lie(G)$. This implies that $\left\Vert ...\right\Vert $ satisfies
the following properties

1. $\left\Vert xy\right\Vert \leq\left\Vert x\right\Vert \left\Vert
y\right\Vert ,x,y\in G.$

2. The sets $\{g\in G|\left\Vert g\right\Vert \leq r\}$ are compact.

3. If $X\in Lie(G)$ satisfies $\theta X=-X$ and $t>0$ then%
\[
\left\Vert \exp tX\right\Vert =\left\Vert \exp X\right\Vert ^{t}.
\]
Since the operator norm and the Hilbert-Schmidt norm are equivalent in a
finite dimensional Hilbert space and if $X\in Lie(N)$, $\wedge^{n}Ad(\exp X) $
has coefficients that are polynomials in $X$ of degree at most $un$ if
$\left(  adX\right)  ^{u+1}=0$. Thus there exists a constant, $C>0$, such that%
\[
\left\Vert \exp X\right\Vert \leq C(1+\left\Vert X\right\Vert )^{\frac{un}{2}%
}.
\]
Also recall the following result of Harish-Chandra \cite{Spher1},c.f.
\cite{RRGI}, Theorem 4.5.3.

\begin{theorem}
There exist $C_{1}$ and $r$ such that
\[
\left\Vert g\right\Vert ^{-1}\leq\Xi(g)\leq C_{1}\left\Vert g\right\Vert
^{-1}(1+\log\left\Vert g\right\Vert )^{r}.
\]

\end{theorem}

We will now prove an estimate that is a consequence of tameness and show how
it relates to the applications asserted in the introduction.

\begin{lemma}
\label{polyest}Let $Z$ \ be a subspace of $Lie(N)$ such that $Lie(N)=Z\oplus
Lie([N,N])$. There exists $s\in\mathbb{R}$ and $\ $for each $\omega\subset G$
a compact subset a constant $C_{\omega}$ such that if $X\in Z$ and if
$g\in\omega$ such that if $d$ is given there exists $C_{d}$ such that
\[
\int_{\lbrack N,N]}\Xi(n\exp Xg)(1+\log\left\Vert n\right\Vert )^{d}dn\leq
C_{\omega}C_{d}(1+\left\Vert X\right\Vert )^{s}.
\]

\end{lemma}

\begin{proof}
In light of the above theorem it is enough to prove that for all $d$%
\[
\int_{\lbrack N,N]}\left\Vert n\exp Xg\right\Vert ^{-1}(1+\log\left\Vert
n\right\Vert )^{d}dn\leq C_{\omega}C_{d}(1+\left\Vert X\right\Vert )^{s}.
\]
If $n\in\lbrack N,N],g\in\omega,X\in Z$ then%
\[
\left\Vert n\right\Vert =\left\Vert n\exp Xgg^{-1}\exp(-X)\right\Vert
\leq\left\Vert n\exp Xg\right\Vert \left\Vert g^{-1}\exp(-X)\right\Vert
\]%
\[
\leq\left\Vert n\exp Xg\right\Vert \left\Vert g^{-1}\right\Vert \left\Vert
\exp(-X)\right\Vert \leq\left\Vert n\exp Xg\right\Vert \left\Vert B_{\omega
}\right\Vert C(1+\left\Vert X\right\Vert )^{\frac{un}{2}}
\]
Since $\left\Vert g^{-1}\right\Vert \leq B_{\omega}$ for $g\in\omega$ and
$\left\Vert \exp(-X)\right\Vert \leq C(1+\left\Vert X\right\Vert )^{\frac
{un}{2}}$ \ So%
\[
\left\Vert n\exp Xg\right\Vert ^{-1}\leq CB_{\omega}\left\Vert n\right\Vert
^{-1}(1+\left\Vert X\right\Vert )^{\frac{n}{2}}.
\]
Let $v$ \ be a non-zero element of $\wedge^{n}\theta Lie(N)$. Then if
$n=ua(n)k$ with $u\in\theta N,a(n)\in A.k\in K$ then%
\[
\left\Vert n^{-1}\right\Vert ^{2}\left\Vert v\right\Vert \geq\left\Vert
n^{-1}v\right\Vert =a(n)^{2\rho}\left\Vert v\right\Vert
\]
so
\[
\left\Vert n^{-1}\right\Vert ^{-1}\leq a(n)^{-\rho}.
\]
Finally, we have (here we are using $\left\Vert g\right\Vert =\left\Vert
g^{-1}\right\Vert $)
\[
\int_{\lbrack N,N]}\left\Vert n\exp Xg\right\Vert ^{-1}(1+\log\left\Vert
n\right\Vert )^{d}dn\leq CB_{\omega}(1+\left\Vert X\right\Vert )^{\frac{n}{2}%
}\int_{[N,N]}\left\Vert n\right\Vert ^{-1}(1+\log\left\Vert n\right\Vert
)^{d}dn
\]%
\[
=CB_{\omega}(1+\left\Vert X\right\Vert )^{\frac{n}{2}}\int_{[N,N]}\left\Vert
n^{-1}\right\Vert ^{-1}(1+\log\left\Vert n^{-1}\right\Vert )^{d}dn
\]%
\[
\leq CB_{\omega}(1+\left\Vert X\right\Vert )^{\frac{n}{2}}\int_{[N,N]}%
a(n)^{-\rho}(1+\log\left\Vert n\right\Vert )^{d}dn.
\]
Completing the proof of the lemma.
\end{proof}

A subgroup $P$ of $G$ is said to be a parabolic subgroup if it contains a
conjugate of the normalizer of $N$ in $G$ and it is its own normalizer. The
unipotent radical of $P$, $N_{P}$, is the maximal normal nilpotent subgroup of
$G$. Thus up to conjugacy we may assume that $P$ contains $N$. If so, we will
call $P$ standard. Assume that $P$ is standard. Let $M_{P}=\theta(P)\cap P$
and $N_{P}^{\ast}=M_{P}\cap N$. The map $M_{P}\times N_{P}\rightarrow P$ given
by multiplication is a diffeomorphism. This implies that $N=N_{P}N^{\ast}$
with unique expression. As in the proof of the preceding lemma the assertions
that are consequences of tameness follow from the next Lemma and that a
function, $f$, on $G$ is said to satisfy the weak inequality if there exist
constants $C$ and $d$ such that%
\[
\left\vert f(g)\right\vert \leq C\Xi(g)(1+\log\left\Vert g\right\Vert )^{d}.
\]

\begin{lemma}
$\int_{[N_{P},N_{P}]}a(n)^{-\rho}(1+\log\left\Vert n\right\Vert )^{d}%
dn<\infty$.
\end{lemma}

\begin{proof}
We note that $[N,N]=[N_{P},N_{P}][N_{P}^{\ast},N_{P}^{\ast}]$ and the Haar
measure on $[N,N]$ is after appropriate normalizations equal to the product
measure. Thus Fubini's Theorem implies that%
\[
\int_{\lbrack N,N]}a(n)^{-\rho}(1+\log\left\Vert n\right\Vert )^{d}%
dn=\int_{[N_{P},N_{P}]\times\lbrack N_{P}^{\ast},N_{P}^{\ast}]}a(nn^{\ast
})^{-\rho}(1+\log\left\Vert nn^{\ast}\right\Vert )^{d}dndn^{\ast}
\]
and for almost all $n^{\ast}\in\lbrack N_{P}^{\ast},N_{P}^{\ast}]$
\[
\overset{}{(\ast)}\int_{[N_{P},N_{P}]}a(nn^{\ast})^{-\rho}(1+\log\left\Vert
nn^{\ast}\right\Vert )^{d}dn<\infty.
\]
Now if $v$ is as in the proof of the previous lemma then
\[
a(nn^{\ast})^{2\rho}=\left\Vert \left(  n^{\ast}\right)  ^{-1}n^{-1}%
v\right\Vert \leq\left\Vert \left(  n^{\ast}\right)  ^{-1}\right\Vert
a(n)^{2\rho}
\]
so
\[
a(n)^{2\rho}\geq\left\Vert n^{\ast}\right\Vert ^{-1}a(nn^{\ast})^{2\rho}
\]
also $\left\Vert nn^{\ast}\right\Vert \leq\left\Vert n\right\Vert \left\Vert
n^{\ast}\right\Vert $ and
\[
\left\Vert n\right\Vert =\left\Vert nn^{\ast}(n^{\ast})^{-1}\right\Vert
\leq\left\Vert n^{\ast}\right\Vert \left\Vert nn^{\ast}\right\Vert
\]
so for fixed $n^{\ast}$,$1+\log\left\Vert nn^{\ast}\right\Vert \geq
C(1+\log\left\Vert n\right\Vert )$ with $C>0$. This if $n^{\ast}$ satisfies
$(\ast)$ then there exists a constant $C_{1}$ such that%
\[
\int_{\lbrack N_{P},N_{P}]}a(n)^{-\rho}(1+\log\left\Vert n\right\Vert
)^{d}dn\leq C_{1}\int_{[N_{P},N_{P}]}a(nn^{\ast})^{-\rho}(1+\log\left\Vert
nn^{\ast}\right\Vert )^{d}dn<\infty.
\]

\end{proof}

We will be using the following simple estimates often in this paper.

\begin{lemma}
\label{inequalities}Let $A_{1},...,A_{n}\geq0,a_{1},...,a_{n}\geq0$ then%
\[
(1+A_{1}+...+A_{n})^{a_{1}+...+a_{n}}\geq(1+A_{1})^{a_{1}}(1+A_{2})^{a_{2}%
}\cdots(1+A_{n})^{a_{n}}
\]
and also
\[
\sum_{i=1}^{n}\log(1+A_{i})\geq\log(1+\sum_{i=1}^{n}A_{i})\geq\frac{1}{n}%
\sum_{i=1}^{n}\log(1+A_{i})\text{.}
\]

\end{lemma}

\begin{proof}%
\[
(1+A_{1}+...+A_{n})^{a_{1}+...+a_{n}}=\prod_{i=1}^{n}(1+A_{1}+...+A_{n}%
)^{a_{i}}\geq\prod_{i=1}^{n}(1+A_{i})^{a_{i}}
\]
proving the first assertion. Using it we have%
\[
(1+A_{1}+...+A_{n})^{n}\geq\prod_{i=1}^{n}(1+A_{i})
\]
implying the right side of the second inequality. The inequality on the
follows from%
\[
\prod_{i=1}^{n}(1+A_{i})\geq1+\sum A_{i}.
\]

\end{proof}

\section{Groups of real rank 1}

Let $G$ be a connected, real reductive group of split rank 1 with finite
center let $K$ be a maximal compact subgroup and $\theta$ the corresponding
Cartan involution, let $P$ be a non-trivial proper parabolic subgroup of $G$
with unipotent radical $N$. Let $G=\bar{N}AK$ be an Iwasawa decomposition of
$G$. If $\mathfrak{n}=Lie(N)$ and $\mathfrak{a}=Lie(A)$ then $\Phi
(P,A)=\{\lambda\}$ if and only if $\mathfrak{n}$ is commutative and
$\Phi(P,A)=\{\lambda,2\lambda\}$ otherwise. If the $\alpha$ root space in
$\mathfrak{n}$ is denoted $\mathfrak{n}_{\alpha}$ then $[\mathfrak{n,n}%
]=\mathfrak{n}_{2\lambda}$. If $X\in\mathfrak{n}_{\lambda}$ and $Y\in
\mathfrak{n}_{2\lambda}$ and if $m_{i\lambda}=\dim\mathfrak{n}_{i\lambda} $
then (c.f. \cite{HarmAn} Lemma 8.10.13)%
\[
a_{\bar{P}}(\exp(X+Y)^{\rho}=((1+\frac{\lambda(H_{\lambda})}{2}\left\Vert
X\right\Vert ^{2})^{2}+2\lambda(H_{\lambda})\left\Vert Y\right\Vert
^{2})^{\frac{m_{\lambda}}{4}+\frac{m_{2\lambda}}{2}}
\]
with $H_{\lambda}$ the element of $\mathfrak{a}$ such that $(H_{\lambda
},h)=\lambda(h),$ $h\in\mathfrak{a}$ and $(...,...)$ is the inner product on
$Lie(G)$ given by $(U,V)=-B(U,\theta V)$ \ with $B$ the killing form \ and
$\left\Vert U\right\Vert ^{2}=(U,U)$, as usual. Thus
\[
\int_{\lbrack N,N]}a_{\bar{P}}(n)^{-\rho_{P}}(1+\log\left\Vert n\right\Vert
)^{r}dn=
\]%
\[
\int_{\mathfrak{n}_{2\lambda}}(1+2\lambda(H_{\lambda})\left\Vert Y\right\Vert
^{2})^{-\frac{m_{\lambda}}{4}-\frac{m_{2\lambda}}{2}}(1+\log(1+(1+\left\Vert
Y\right\Vert ^{2}))^{r}dY<\infty
\]
for all $r$ since $m_{\lambda}>1$ if $\mathfrak{n}_{2\lambda}\neq0$ .

\section{$Sp_{4}(\mathbb{R})$ and $Sp_{4}(\mathbb{C})$}

Let $F=\mathbb{R}$ or $\mathbb{C}$. Set%
\[
L=\left[
\begin{array}
[c]{cc}%
0 & 1\\
1 & 0
\end{array}
\right]
\]
and%
\[
J=\left[
\begin{array}
[c]{cc}%
0 & L\\
-L & 0
\end{array}
\right]  .
\]
We realize $G=Sp_{4}(F)$ as
\[
G=\{g\in GL(2.F)|gJg^{-1}\}.
\]
We choose $P$ to be the upper triangular elements of $G$. Thus the lie algebra
of $N$ is
\[
\mathfrak{n}=\left\{  \left[
\begin{array}
[c]{cccc}%
0 & x & y & z\\
0 & 0 & w & y\\
0 & 0 & 0 & -x\\
0 & 0 & 0 & 0
\end{array}
\right]  |x,y,z,w\in F\right\}  .
\]
So%
\[
\lbrack\mathfrak{n},\mathfrak{n}]=\left\{  \left[
\begin{array}
[c]{cccc}%
0 & 0 & y & z\\
0 & 0 & 0 & y\\
0 & 0 & 0 & 0\\
0 & 0 & 0 & 0
\end{array}
\right]  |y,z\in F\right\}  .
\]
So
\[
\lbrack N_{o},N_{o}]=\left\{  n(y,z)=\left[
\begin{array}
[c]{cccc}%
1 & 0 & y & z\\
0 & 1 & 0 & y\\
0 & 0 & 1 & 0\\
0 & 0 & 0 & 1
\end{array}
\right]  |y,z\in F\right\}  .
\]
If $g\in G$ and if $g_{1},g_{2}$ are the last two columns of $g^{-1}$ then%
\[
a(g)^{\rho_{P_{o}}}=\left\Vert g_{2}\right\Vert ^{\dim_{\mathbb{R}}%
F}\left\Vert g_{1}\wedge g_{2}\right\Vert ^{\dim_{\mathbb{R}}F}
\]
so if $n(y,z)\in\lbrack N,N]$ then%
\[
a(n)^{\rho}=(1+|y|^{2}+|z|^{2})^{\frac{\dim_{\mathbb{R}}F}{2}}((1+\left\vert
y\right\vert ^{2})^{2}+|z|^{2})^{\frac{\dim_{\mathbb{R}}F}{2}}
\]
Observing that%
\[
(1+|y|^{2}+|z|^{2})\geq(1+\left\vert y\right\vert ^{2})^{\frac{1}{3}%
}(1+\left\vert z\right\vert ^{2})^{\frac{2}{3}}
\]
and%
\[
((1+\left\vert y\right\vert ^{2})^{2}+|z|^{2})\geq(1+\left\vert y\right\vert
^{2})^{2\frac{1}{2}}(1+\left\vert z\right\vert ^{2})^{\frac{1}{2}}.
\]
So
\[
a(n)^{\rho}\geq((1+\left\vert y\right\vert ^{2})^{\frac{4}{3}}(1+\left\vert
z\right\vert ^{2})^{\frac{7}{6}})^{\frac{\dim_{\mathbb{R}}F}{2}}
\]
From this it is easy to see that%
\[
\int_{\lbrack N_{o},N_{o}]}a(n)^{\rho}(1+\rho(\log a(n)))^{r}dn<\infty
\]
for all $r\in\mathbb{R}$ .

\section{$SO(n,2)$}

In this section $G=SO(n+2,2)$ which we realize as follows: Let%
\[
L=\left[
\begin{array}
[c]{cc}%
0 & 1\\
1 & 0
\end{array}
\right]
\]
and%
\[
H=\left[
\begin{array}
[c]{ccc}%
0 & 0 & L\\
0 & I & 0\\
L & 0 & 0
\end{array}
\right]
\]
where $I$ denotes the $n\times n$ identity matrix and $G$ is the group of all
elements of $SL(n+4,\mathbb{R})$ such that
\[
gHg^{T}=H.
\]
Let $K=G\cap SO(n+4)$ then $K$ is isomorphic with $S(SO(2)\times SO(n+2))$.
$\mathfrak{n=}Lie(N)$ is the group of elements of $M_{n+4}(\mathbb{R})$ of the
form%
\[
\left[
\begin{array}
[c]{ccc}%
\begin{array}
[c]{cc}%
0 & x\\
0 & 0
\end{array}
& Y &
\begin{array}
[c]{cc}%
z & \\
& -z
\end{array}
\\
0 & 0 & -Y^{T}L\\
0 & 0 &
\begin{array}
[c]{cc}%
0 & -x\\
0 & 0
\end{array}
\end{array}
\right]
\]
with $Y$ of size $2\times n.$ One checks that $[\mathfrak{n},\mathfrak{n}]$ is
the space of matrices%
\[
X=\left[
\begin{array}
[c]{ccccc}%
0 & 0 & y & z & 0\\
0 & 0 & 0 & 0 & -z\\
0 & 0 & 0 & 0 & -y^{T}\\
0 & 0 & 0 & 0 & 0\\
0 & 0 & 0 & 0 & 0
\end{array}
\right]
\]
with $y$ of size $1\times n$. With this notation we have%
\[
\exp(X)=\left[
\begin{array}
[c]{ccccc}%
1 & 0 & y & z & -\frac{\left\Vert y\right\Vert ^{2}}{2}\\
0 & 1 & 0 & 0 & -z\\
0 & 0 & I & 0 & -y^{T}\\
0 & 0 & 0 & 1 & 0\\
0 & 0 & 0 & 0 & 1
\end{array}
\right]
\]
and one calculates

\begin{lemma}
With $X$ as above%
\[
a(\exp(X))^{2\rho}=(\left(  1+\frac{\left\Vert y\right\Vert ^{2}}{2}\right)
^{2}+z^{2})(1+z^{2}+\frac{\left\Vert y\right\Vert ^{2}}{2})^{n}.
\]
Thus,%
\[
a(\exp(X))^{-\rho}\leq(1+\frac{\left\Vert y\right\Vert ^{2}}{2})^{-\frac{n}%
{2}-\frac{1}{12}}(1+z^{2})^{-\frac{1}{2}-\frac{1}{12}}.
\]

\end{lemma}

\begin{proof}
We note that if
\[
m=\left[
\begin{array}
[c]{ccccc}%
1 & 0 & 0 & 0 & 0\\
0 & 1 & 0 & 0 & 0\\
0 & 0 & U & 0 & 0\\
0 & 0 & 0 & 1 & 0\\
0 & 0 & 0 & 0 & 1
\end{array}
\right]
\]
with $U\in SO(n)$ then $m\in G$ and
\[
a(m\exp(X)m^{-1})=\left[
\begin{array}
[c]{ccccc}%
1 & 0 & yU & z & -\frac{\left\Vert y\right\Vert ^{2}}{2}\\
0 & 1 & 0 & 0 & -z\\
0 & 0 & I & 0 & -U^{T}y^{T}\\
0 & 0 & 0 & 1 & 0\\
0 & 0 & 0 & 0 & 1
\end{array}
\right]
\]
thus we may assume that
\[
\exp X=\left[
\begin{array}
[c]{ccccc}%
1 & 0 & re_{1} & z & -\frac{r^{2}}{2}\\
0 & 1 & 0 & 0 & -z\\
0 & 0 & I & 0 & -re_{1}^{T}\\
0 & 0 & 0 & 1 & 0\\
0 & 0 & 0 & 0 & 1
\end{array}
\right]
\]
with $e_{i}$ the $1\times n$ matrix with a $1$ in the $i\mathrm{th}$ position
and zeros elsewhere. If $a$ is a diagonal matrix with diagonal entries
$a_{1},...,a_{m}$ then we define $\mu_{j}(a)=a_{1}\cdots a_{j}$ for $j\leq m$.
On $\wedge^{j}\mathbb{R}^{n+4}$ we put the inner product that makes $e_{i_{1}%
}\wedge\cdots\wedge e_{i_{j}}$ for $1\leq i_{1}<...<i_{j}\leq n+4$ an
orthonormal basis then if $g\in G$%
\[
a(g)^{\mu_{j}}=\left\Vert e_{1}g\wedge\cdots\wedge e_{j}g\right\Vert .
\]
Now
\[
e_{i}\exp X=\left\{
\begin{array}
[c]{c}%
e_{1}+re_{3}+ze_{n+3}-\frac{r^{2}}{2}e_{n+4},\mathrm{if\ }i=1\\
e_{2}-ue_{n+4},\mathrm{if\ }i=2\\
e_{3}-re_{n+4},\mathrm{if\ }i=3\\
e_{i},\mathrm{if\ }i>3
\end{array}
\right.
\]
we also note that if $V_{i}$ is the span of $\{e_{j}|j\neq i\}$ then the map
$\wedge^{j}V_{i}\rightarrow\wedge^{j+1}\mathbb{R}^{n+4}$ given by $u\mapsto
u\wedge e_{i}$ is isometric if $j<n+4$. \ This implies that if $3<i<n+3$ then
\[
a(\exp X)^{\mu_{i}}=a(\exp X)^{\mu_{3}}.
\]
We leave it to the reader to show that%
\[
a(\exp X)^{2\mu_{1}}=\left(  1+\frac{r^{2}}{2}\right)  ^{2}+z^{2}
\]
and%
\[
a(\exp X)^{2\mu_{2}}=a(\exp X)^{2\mu_{3}}=(1+z^{2}+\frac{r^{2}}{2})^{2}
\]
thus%
\[
a(\exp X)^{2\mu_{j}}=(1+z^{2}+\frac{r^{2}}{2})^{2}\text{ for }3<j<n+3.
\]
Now if $n=2k$ then $\rho=\mu_{1}+...+\mu_{k+1}$ and if $n=2k+1$ then $\rho
=\mu_{1}+...+\mu_{k+1}+\frac{1}{2}\mu_{k+2}$. Thus using the observations
above we have in both cases
\[
a(\exp X)^{2\rho}=(\left(  1+\frac{r^{2}}{2}\right)  +z^{2})(1+z^{2}%
+\frac{r^{2}}{2})^{n}.
\]
The inequality is proved by observing that if $0<\varepsilon<1$ and
$0<\delta<1$ then Lemma \ref{inequalities}%
\[
(\left(  1+\frac{r^{2}}{2}\right)  ^{2}+z^{2})(1+z^{2}+\frac{r^{2}}{2}%
)^{n}\geq\left(  1+\frac{r^{2}}{2}\right)  ^{2\varepsilon+n(1-\delta)}%
(1+z^{2})^{1-\varepsilon+n\delta}.
\]
Taking $\varepsilon=\frac{1}{3}$ and $\delta=\frac{1}{2n}$ yields the inequality
\end{proof}

So $SO(n+2,2)$ is tame. Note that the identity component of $SO(4,2)$ is
locally isomorphic with $SU(2,2)$ thus $SU(2,2)$ is tame. We also note that
the identity component of $SO(3,2)$ \ is locally isomorphic with
$Sp_{4}(\mathbb{R})$ which reproves that it is tame.

\section{Harish-Chandra's convergence theorem for $SL(n,F),F=\mathbb{R}$ or
$\mathbb{C}$}

In this section we will give a proof of Harish-Chandra's convergence theorem
for $SL(n,F)$, which will serve as a template for our proof of the tameness of
$SL(n,F)$.

Here $G_{n}=SL(n,F),F=\mathbb{R}$ or $\mathbb{C},K_{n}$ is $SO(n)$ if
$F=\mathbb{R}$ or $SU(n)$ if $F=\mathbb{C}$, $N_{n}$ is the group of upper
triangular matrices in $G_{n}$ with ones on the main diagonal and $A_{n}$ is
the group of diagonal matrices in $G_{n}$ with real positive entries. $\bar
{N}_{n}=N_{n}^{T}$ and $G_{n}=\bar{N}_{n}A_{n}K_{n}$ the corresponding Iwasawa
decomposition. Write $g=\bar{n}(g)a_{n}(g)k_{n}(g)$ for $g\in G_{n}$ as usual.
If $v\in N_{n}$ then writing $v$ out in terms of its columns%
\[
v=[v_{1}v_{2}...v_{n}],v_{1}=[1,0,...,0]^{T},v_{i}=[x_{\binom{i-1}{2}%
+1},...,x_{\binom{i}{2}},1,0..,0]^{T},...
\]
Let $\Lambda_{i}$ be the basic highest weights (that is the highest weights of
the representations $\wedge^{i}\mathbb{R}^{n}$ respectively) then if $\xi
_{i}=e_{n-i+1}\wedge\cdots\wedge e_{n}$ and if $g\in SL(n,F)$ then
\[
a_{n}(g)^{\Lambda_{i}}=\left\Vert g^{-1}\xi_{i}\right\Vert
\]
so
\[
a_{n}(v^{-1})^{\Lambda_{i-1}}=\left\Vert v_{i}\wedge\cdots\wedge
v_{n}\right\Vert
\]
for $i=2,...,n$. Thus%
\[
a_{n}(v^{-1})^{2\rho}=\prod_{i=2}^{n}\left\Vert v_{i}\wedge\cdots\wedge
v_{n}\right\Vert ^{2\dim_{\mathbb{R}}F}.
\]
This easily implies

\begin{lemma}
\label{keyformula}If
\[
v=\left[
\begin{array}
[c]{cc}%
I & y\\
0 & 1
\end{array}
\right]
\]
with
\[
y=[y_{1},...,y_{n-1}]^{T}
\]
then%
\[
a_{n}(v)^{2\rho}=\prod_{i=1}^{n-1}(1+\sum_{j=1}^{i}y_{i}^{2})^{\dim
_{\mathbb{R}}F}.
\]

\end{lemma}

Note that $N_{n}=N^{\ast}V$ with
\[
N^{\ast}=\left[
\begin{array}
[c]{cc}%
N_{n-1} & 0\\
0 & 1
\end{array}
\right]  ,V=\left[
\begin{array}
[c]{cc}%
I_{n-1} & F^{n-1}\\
0 & 1
\end{array}
\right]  .
\]
Also if $v^{\ast}\in N^{\ast}$ then
\[
v^{\ast}=\bar{v}^{\ast}a^{\ast}k^{\ast}
\]
with the constituents giving the iwasawa decomposition in $\left[
\begin{array}
[c]{cc}%
G_{n-1} & 0\\
0 & 1
\end{array}
\right]  $, Define for $v^{\ast}\in N^{\ast}$ $w(v^{\ast})\in N_{n-1}$ by%
\[
v^{\ast}=\left[
\begin{array}
[c]{cc}%
w(v^{\ast}) & 0\\
0 & 1
\end{array}
\right]
\]
Note that%
\[
a_{\bar{P}_{n}}(v^{\ast})^{\rho_{n}}=a_{\bar{P}_{n-1}}(w(v^{\ast}%
))^{\rho_{n-1}},a_{\bar{P}_{n}}(v^{\ast})^{\Lambda_{n-1}}=1.
\]
If $v^{\ast}\in N^{\ast},v_{1}\in V$ then
\[
v^{\ast}v_{1}=\bar{n}(v^{\ast})a_{n}(v^{\ast})k_{n}(v^{\ast})v_{1}=\bar
{n}(v^{\ast})a_{n}(v^{\ast})k_{n}(v^{\ast})v_{1}k_{n}(v^{\ast})^{-1}k(v^{\ast
}).
\]
If
\[
k_{n}(v^{\ast})v_{1}k_{n}(v^{\ast})^{-1}=\bar{v}a_{n}(k_{n}(v^{\ast}%
)v_{1}k(v^{\ast})^{-1})k_{n}
\]
is its Iwasawa decomposition then%
\[
v^{\ast}v_{1}=\bar{n}(v^{\ast})(a_{n}(v^{\ast})\bar{v}a_{n}(v^{\ast}%
)^{-1})a_{n}(v^{\ast})a_{n}(k_{n}(v^{\ast})v_{1}k_{n}(v^{\ast})^{-1}%
)k_{n}(v^{\ast})k
\]
So%
\[
a_{n}(v^{\ast}v_{1})=a_{n}(v^{\ast})a_{n}(k_{n}(v^{\ast})v_{1}k_{n}(v^{\ast
})^{-1}).
\]

We now show how this material gives a simple inductive proof of a slight
strengthening of Harish-Chandra's result in the case of $SL(n,F)$.

\begin{theorem}
If $n\geq2$ and $\varepsilon>0$ then%
\[
\int_{N_{n}}a_{n}(v)^{-\rho_{n}}(1+\rho(\log a_{n}(v))^{-\binom{n}%
{2}-\varepsilon}dv<\infty.
\]

\end{theorem}

\begin{proof}
We prove the result by induction on $n$. If $F=\mathbb{C}$ then we identify
$F$ with $\mathbb{R}^{2}$ and $dx$ will denote Labesgue measure. If $n=2$ then
we are looking at%
\[
\int_{F}\frac{dx}{(1+\left\vert x\right\vert ^{2})^{\frac{\dim_{\mathbb{R}}%
F}{2}}(1+\log(1+\left\vert x\right\vert ^{2}))^{1+\varepsilon}}
\]
which finite if $\varepsilon>0$ (We will give simple argument after the proof
is completed). So assume the result for $n-1$. We have%
\[
\int_{N_{n}}a_{n}(v)^{-\rho_{n}}(1+\rho_{n}(\log a(v))^{-\binom{n}%
{2}-\varepsilon}dv=\int_{N^{\ast}}a_{n-1}(w(v^{\ast}))^{-\rho_{n-1}}\times
\]%
\[
\int_{V}a_{n}(k_{n}(v^{\ast})v_{1}k_{m}(v^{\ast})^{-1})^{-\rho_{n}}%
(1+\rho_{n-1}(\log a_{n-1}(w(v^{\ast}))+\rho_{n}(\log a_{n}(k_{n}(v^{\ast
})v_{1}k_{n}(v^{\ast})^{-1}))^{^{-\binom{n}{2}-\varepsilon}}dv^{\ast}dv
\]
We note that the map $v_{1}\rightarrow gv_{1}g^{-1}$ preserves $V$ if
$g\in\left[
\begin{array}
[c]{cc}%
G_{n-1} & 0\\
0 & 1
\end{array}
\right]  $ and preserves the invariant measure on $V$. Thus the integral
becomes%
\[
\int_{N^{\ast}}a_{n-1}(w(v^{\ast}))^{-\rho_{n-1}}\int_{V}a_{n}(v_{1}%
)^{-\rho_{n}}(1+\rho_{n-1}(\log a_{n-1}(w(v^{\ast}))+\rho_{n}(\log a_{n}%
(v_{1}))^{^{-\binom{n}{2}-\varepsilon}}dv^{\ast}dv
\]%
\[
\leq\int_{N^{\ast}}a_{\bar{P}_{n-1}}(w(v^{\ast}))^{-\rho_{n-1}}(1+\rho
_{n-1}(\log a_{\bar{P}_{n-1}}(w(v^{\ast}))^{-\binom{n-1}{2}-\frac{\varepsilon
}{2}}dv^{\ast}\times
\]%
\[
\int_{V}a_{n}(v_{1})^{-\rho_{n}}(1+\rho_{n}(\log a_{n}(v_{1}))^{-n+1-\frac
{\varepsilon}{2}}dv_{1}.
\]
The first integral is covered by the inductive hypothesis. To analyze the
second integral, the above lemma implies that it is equal to
\[
\int_{F^{n-1}}(1+\sum_{i=1}^{n-1}\log(1+\sum_{j=1}^{i}|y_{i}|^{2}%
))^{-n+1-\frac{\varepsilon}{2}}\prod_{i=1}^{n-1}(1+\sum_{j=1}^{i}|y_{i}%
|^{2})^{-\frac{\dim_{\mathbb{R}}F}{2}}dy.
\]
since we can write
\[
v=\left[
\begin{array}
[c]{cc}%
I_{n-1} & y^{T}\\
0 & 1
\end{array}
\right]  ,y=[y_{1}y_{2}...y_{n-1}].
\]
This integral converges since it is less than or equal to%
\[
\left(  \int_{F}\frac{dx}{(1+|x|^{2})^{\frac{\dim_{\mathbb{R}}F}{2}}%
(1+\log(1+|x\}^{2}))^{1+\frac{\varepsilon}{2(n-1)}}}\right)  ^{n-1}<\infty.
\]

\end{proof}

\begin{lemma}
If $\varepsilon>0$ then
\[
\int_{F}\frac{dx}{(1+|x|^{2})^{\frac{\dim_{\mathbb{R}}F}{2}}(1+\log
(1+|x|^{2}))^{1+\varepsilon}}<\infty.
\]

\end{lemma}

\begin{proof}
Integrating in polar coordinates the integral is%
\[
2\pi^{\dim_{\mathbb{R}}F-1}\int_{0}^{\infty}\frac{r^{\dim_{\mathbb{R}}F-1}%
dr}{(1+r^{2})^{\frac{\dim_{\mathbb{R}}F}{2}}(1+\log(1+r^{2}))^{1+\varepsilon}%
}.
\]
Which we can write as%
\[
2\pi^{\dim_{\mathbb{R}}F-1}\int_{0}^{2}r^{\dim_{\mathbb{R}}F-1}\frac
{dr}{(1+r^{2})^{\frac{\dim_{\mathbb{R}}F}{2}}(1+\log(1+r^{2}))^{1+\varepsilon
}}
\]%
\[
+2\pi^{\dim_{\mathbb{R}}F-1}\int_{2}^{\infty}r^{\dim_{\mathbb{R}}F-1}\frac
{dr}{(1+r^{2})^{\frac{\dim_{\mathbb{R}}F}{2}}(1+\log(1+r^{2}))^{1+\varepsilon
}}.
\]
The first integral is clearly finite. As for the second it is less than or
equal to a constant times%
\[
\int_{2}^{\infty}\frac{dr}{r\log(r)^{1+\varepsilon}}=\int_{\log2}^{\infty
}\frac{dr}{r^{1+\varepsilon}}<\infty.
\]

\end{proof}

\section{$SL(n,F),F=\mathbb{R}$ or $\mathbb{C}$}

We retain the notation of the previous section. That the group $SL(n,F)$ is
tame is an immediate consequence of the following result the proof of which
will be the topic of the rest of this section. The extra factor is necessary
in its inductive proof.

\begin{theorem}
If $n\geq3$ if $0\leq\alpha<1$ then%
\[
\int_{\lbrack N_{n},N_{n}]}a_{n}(v)^{-\rho_{n}}a_{n}(v)^{\alpha\dim
_{\mathbb{R}}F\Lambda_{n-1}}(1+\rho_{n}(\log(a_{n}(v)))^{r}dv<\infty
\]
for all $r$.
\end{theorem}

We first note that
\[
\lbrack N,N]\cap N^{\ast}=[N^{\ast},N^{\ast}]=\left[
\begin{array}
[c]{cc}%
\lbrack N_{n-1},N_{n-1}] & 0\\
0 & 1
\end{array}
\right]
\]
Also, if%
\[
v_{1}\in\lbrack N,N]\cap V
\]
then%
\[
v_{1}=\left[
\begin{array}
[c]{cc}%
I_{n-1} & y^{T}\\
0 & 1
\end{array}
\right]  ,y=[y_{1},...,y_{n-2,}0,1].
\]
Note that the map
\[
\left(  \lbrack N,N]\cap N^{\ast}\right)  \times([N,N]\cap V)\rightarrow
\lbrack N,N]
\]
given by multiplication defines a measure preserving diffeomorphism if the
Haar measures on the three groups are properly normalized.

If $v^{\ast}\in N^{\ast}$ with
\[
v^{\ast}=\left[
\begin{array}
[c]{cc}%
w(v^{\ast}) & 0\\
0 & 1
\end{array}
\right]  ,w(v^{\ast})\in N_{n-1}
\]
then
\[
k_{n}(v^{\ast})=\left[
\begin{array}
[c]{cc}%
k_{n-1}(w(v^{\ast})) & 0\\
0 & 1
\end{array}
\right]  .
\]
So we have%
\[
k_{n}(v^{\ast})v_{1}k_{n}(v^{\ast})^{-1}=\left[
\begin{array}
[c]{cc}%
I_{n-1} & k_{n-1}(w(v^{\ast}))y^{T}\\
0 & 1
\end{array}
\right]
\]
Set $V_{1}=V\cap\lbrack N,N],$as in the proof of Harish-Chandra's convergence
theorem we have
\[
\int_{\lbrack N_{n},N_{n}]}a_{n}(v)^{-\rho_{n}}a_{n}(v)^{\alpha\dim
_{\mathbb{R}}F\Lambda_{n-1}}(1+\rho(\log(a_{n}(v)))^{r}dv=
\]%
\[
\int_{N^{\ast}\cap\lbrack N_{n},N_{n}]}a_{n-1}(w(v^{\ast}))^{-\rho_{n-1}}%
\int_{V_{1}}a_{n}(k(v^{\ast})vk(v^{\ast})^{-1})^{-(\rho_{n}-\alpha
\dim_{\mathbb{R}}F\Lambda_{n-1})}\times
\]%
\[
(1+\rho_{n-1}(\log a_{n-1}(w(v^{\ast})))+\rho_{n}(\log a_{n}(k(v^{\ast
})vk(v^{\ast})^{-1}))^{r}dv^{\ast}dv\leq
\]%
\[
\int_{N^{\ast}}a_{n-1}(w(v^{\ast}))^{-\rho_{n-1}}(1+\rho_{n-1}(\log
a_{n}(w(v^{\ast})))^{r}\times
\]%
\[
\int_{V_{1}}a_{n}(k(v^{\ast})vk(v^{\ast})^{-1})^{-(\rho_{n}-\alpha
\dim_{\mathbb{R}}F\Lambda_{n-1})}(1+\rho_{n}(\log a_{n}(k(v^{\ast}%
)v_{1}k(v^{\ast})^{-1}))^{r}dv_{1}dv^{\ast}.
\]
To continue the argument we must analyze the following integral with $v^{\ast
}\in\lbrack N^{\ast},N^{\ast}]$ fixed%
\[
\int_{V_{1}}a_{n}(k(v^{\ast})v_{1}k(v^{\ast})^{-1})^{-(\rho_{n}-\alpha
\dim_{\mathbb{R}}F\Lambda_{n-1})}(1+\rho_{n}(\log a_{n}(k(v^{\ast}%
)v_{1}k(v^{\ast})^{-1}))^{r}dv_{1}
\]
with $v^{\ast}\in\lbrack N^{\ast},N^{\ast}]$ fixed. We note that if
\[
u=\left[
\begin{array}
[c]{cc}%
I & u^{\prime}\\
0 & 1
\end{array}
\right]
\]
with%
\[
u/=[u_{1},u_{2},...,u_{n-1}]
\]
then
\[
\rho_{n}(\log(a_{n}(u))=\frac{\dim_{\mathbb{R}}F}{2}\sum_{i=1}^{n-1}%
\log(1+\sum_{j=1}^{i}|u_{j}|^{2})\leq\frac{\dim_{\mathbb{R}}F}{2}%
(n-1)\log(1+\sum_{i=1}^{n-1}\left\vert u_{i}\right\vert ^{2}).
\]
We write%
\[
k_{n-1}(w(v^{\ast}))=\left[
\begin{array}
[c]{cc}%
A & b\\
c & d
\end{array}
\right]
\]
as above. Noting that if $v^{\ast}\in N^{\ast}$ then Lemma \ref{detformula}
implies that, since the \textquotedblleft$L$\textquotedblright\ in the lemma
is the identity,
\[
\left\vert \det(A)\right\vert =a_{n-1}(w(v^{\ast}))^{-\Lambda_{n-2}}.
\]
Writing $Ay=[\left(  Ay\right)  _{1}(Ay)_{2}...(Ay)_{n-2}]^{T}$ for
$y=[y_{1}...y_{n-2}]^{T}$ and%
\[
v_{1}=\left[
\begin{array}
[c]{cc}%
I &
\begin{array}
[c]{c}%
y^{\prime}\\
0
\end{array}
\\
0 & 1
\end{array}
\right]
\]%
\[
\int_{V_{1}}a_{n}(k(v^{\ast})v_{1}k(v^{\ast})^{-1})^{-(\rho_{n}-\alpha
\dim_{\mathbb{R}}F\Lambda_{n-1})}(1+\rho_{n}(\log a_{n}(k(v^{\ast}%
)v_{1}k(v^{\ast})^{-1}))^{r}dv_{1}\leq
\]%
\[
(n-1)^{r}\int_{F^{n-2}}\frac{(1+\log(1+\sum_{j=1}^{n-2}\left\vert (Ay^{\prime
})_{j}\right\vert ^{2}+\left\vert cy^{\prime}\right\vert ^{2})^{r}%
(1+\sum_{j=1}^{n-2}\left\vert (Ay^{\prime})_{j}\right\vert ^{2}+\left\vert
cy^{\prime}\right\vert ^{2})^{\dim_{\mathbb{R}}F\frac{\alpha}{2}}dy^{\prime}%
}{(1+\sum_{j=1}^{n-2}\left\vert (Ay^{\prime})_{j}\right\vert ^{2}+\left\vert
cy^{\prime}\right\vert ^{2})^{\frac{\dim_{\mathbb{R}}F}{2}}\prod_{i=1}%
^{n-2}(1+\sum_{j=1}^{i}\left\vert (Ay^{\prime})_{j}\right\vert )^{\frac
{\dim_{\mathbb{R}}F}{2}}}.
\]
After the change of variables $Ay^{\prime}\rightarrow y^{\prime}$ the integral
becomes%
\[
(n-1)^{r}\left\vert \det A\right\vert ^{-\dim_{\mathbb{R}}F}\times
\]%
\[
\int_{F^{n-2}}\frac{(1+\log(1+\sum_{j=1}^{n-2}\left\vert y_{j}\right\vert
^{2}+|cA^{-1}y^{\prime}|^{2}))^{r}(1+\sum_{j=1}^{n-2}\left\vert y_{j}%
\right\vert ^{2}+|cA^{-1}y^{\prime}|^{2})^{\dim_{\mathbb{R}}F\frac{\alpha}{2}%
}dy^{\prime}}{(1+\sum_{j=1}^{n-2}\left\vert y_{j}\right\vert ^{2}%
+|cA^{-1}y^{\prime}|^{2})^{\frac{\dim_{\mathbb{R}}F}{2}}\prod_{i=1}%
^{n-2}(1+\sum_{j=1}^{n-2}\left\vert y_{j}\right\vert ^{2})^{\frac
{\dim_{\mathbb{R}}F}{2}}}.
\]
Let $0<\delta<1-\alpha$. There exists $B_{\delta,r}$ such that%
\[
(1+\log(1+\sum_{j=1}^{n-2}y_{j}^{2}+(cA^{-1}y^{\prime})^{2}))^{r}\leq
B_{\delta,r}(1+\sum_{j=1}^{i}y_{j}^{2}+(cA^{-1}y^{\prime})^{2})^{\frac{\delta
}{2}}.
\]
So if we set $\bar{u}=cA^{-1}\in\mathbb{C}^{n-2}$ then since $0<a+\delta<1$ we
see that Lemma \ref{mainLemma} implies that%
\[
\int_{V_{1}}a_{\bar{P}_{n}}(k(v^{\ast})vk(v^{\ast})^{-1})^{-(\rho_{n}%
-\alpha\dim_{\mathbb{R}}F\Lambda_{n})}(1+\rho_{n}(\log a_{\bar{P}_{n}%
}(k(v^{\ast})v_{1}k(v^{\ast})^{-1}))^{r}dv_{1}\leq
\]%
\[
(n-1)^{r}\left\vert \det A\right\vert ^{-\dim_{\mathbb{R}}F}B_{\delta,r}%
\int_{\mathbb{R}^{n-2}}\varphi_{1-\alpha-\delta,u}(x)^{-\frac{\dim
_{\mathbb{R}}F}{2}}dx\leq
\]%
\[
(n-1)^{r}\left\vert \det A\right\vert ^{-\dim_{\mathbb{R}}F}B_{\delta
,r}C_{\varepsilon,\alpha,m}(1+\left\Vert u\right\Vert ^{2})^{-\frac
{(1-\varepsilon)(1-\alpha-\delta)}{2}\dim_{\mathbb{R}}F}
\]
for all $0<\varepsilon<1$. Lemma \ref{simple} implies that
\[
cA^{-1}==-\frac{b^{\ast}}{\overline{d}}.
\]
So%
\[
\left\Vert u\right\Vert ^{2}=\frac{1-|d|^{2}}{|d|^{2}}.
\]
And we have seen that if $w(n^{\ast})^{-1}$ has last column $[x_{1}%
...x_{n-3}01]^{T}$ then
\[
\left\vert d\right\vert ^{2}=\frac{1}{1+\sum_{i=1}^{n-3}|x_{i}|^{2}}
\]
Thus
\[
\left\Vert u\right\Vert ^{2}=\sum_{i=1}^{n-3}\left\vert x_{i}\right\vert
^{2}.
\]
Hence%
\[
1+\left\Vert u\right\Vert ^{2}=a_{n-1}(w(n^{\ast}))^{2\Lambda_{n-2}}.
\]
Also $\left\vert \det A\right\vert =\left\vert d\right\vert ^{-1}$ so%
\[
\int_{V_{1}}a_{n}(k(v^{\ast})v_{1}k(v^{\ast})^{-1})^{-(\rho_{n}-\alpha
\dim_{\mathbb{R}}F\Lambda_{n-1})}(1+\rho_{n}(\log a_{\bar{P}_{n}}(k(v^{\ast
})v_{1}k(v^{\ast})^{-1}))^{r}dv_{1}\leq
\]%
\[
(n-1)^{r}B_{\delta,r}C_{\varepsilon,\alpha,m}a_{n-1}(w(n^{\ast}%
))^{(1-(1-\varepsilon)(1-\alpha-\delta))\dim_{\mathbb{R}}F\Lambda_{n-2}}.
\]
Hence%
\[
\int_{\lbrack N_{n},N_{n}]}a_{n}(v)^{-\rho_{n}}a_{n}(v)^{\alpha\dim
_{\mathbb{R}}F\Lambda_{n-1}}(1+\rho_{n}(\log(a_{n}(v)))^{r}dv\leq
\]%
\[
(n-1)^{r}B_{\delta,r}C_{\varepsilon,\alpha,m}\int_{[N_{n-1},N_{n-1}]}%
a_{n-1}(z)^{-\rho_{n-1}}\times
\]%
\[
a_{n-1}(z)^{\left(  1-(1-\varepsilon)(1-\alpha-\delta))\right)  \dim
_{\mathbb{R}}F\Lambda_{n-2}}\log(1+\rho_{n-1}(\log(a_{n-1}(z)))^{r}dz
\]
which is finite by the inductive hypothesis, since $0<(1-\varepsilon
)(1-\alpha-\delta)<1$.

\section{Appendices}

In the following appendices $F$ denotes either $\mathbb{R}$ or $\mathbb{C}$.

\subsection{Some linear algebra}

The purpose of this appendix is to point about two simple linear algebra lemmas.

\begin{lemma}
\label{detformula}If $g\in SL(n,F)$ with
\[
g=\left[
\begin{array}
[c]{cc}%
L & u\\
v & w
\end{array}
\right]  ,
\]
if the last column of $g^{-1}$ is
\[
x=[x_{1},...,x_{n}]^{T},
\]
and if
\[
k_{n}(g)=\left[
\begin{array}
[c]{cc}%
A & b\\
c & d
\end{array}
\right]
\]
with $A,L$ matrices of size $(n-1)\times(n-1)$ then%
\[
\det A=\frac{\det L}{\sqrt{\sum_{i=1}^{n}\left\vert x_{i}\right\vert ^{2}}%
}=\left(  \det L\right)  a_{\bar{P}_{n}}(g)^{-\Lambda_{n-1}}.
\]

\end{lemma}

\begin{proof}
We write the Iwasawa decomposition
\[
g=\left[
\begin{array}
[c]{cc}%
\bar{n} & 0\\
z & 1
\end{array}
\right]  \left[
\begin{array}
[c]{cc}%
a & 0\\
0 & \det(a)^{-1}%
\end{array}
\right]  \left[
\begin{array}
[c]{cc}%
A & b\\
c & d
\end{array}
\right]
\]%
\[
=\left[
\begin{array}
[c]{cc}%
\bar{n}a & 0\\
za & \det(a)^{-1}%
\end{array}
\right]  \left[
\begin{array}
[c]{cc}%
A & b\\
c & d
\end{array}
\right]  =\left[
\begin{array}
[c]{cc}%
\bar{n}aA & \bar{n}ab\\
uaA+\det(a)^{-1}c & uab+\det(a)^{-1}d
\end{array}
\right]  .
\]
Thus
\[
\det L=\det a\det A.
\]
To compute $\det a$, note that
\[
\lbrack x_{1},...,x_{n}]^{T}=g^{-1}e_{n}=k_{n}(g)^{-1}\det(a)e_{n}.
\]
Thus%
\[
\det(a)=\left\Vert [x_{1},...,x_{n}]\right\Vert =\sqrt{\sum\left\vert
x_{i}\right\vert ^{2}}=a_{\bar{P}}(g)^{\Lambda_{n-1}}.
\]
The lemma follows.
\end{proof}

If $a\in\mathbb{R}$ then we write $\bar{a}=a$. If $R$ is a $p\times q$ matrix
with coefficients in $F$ then $R^{\ast}$ will mean $\bar{R}^{T}$ .

\begin{lemma}
\label{simple}If
\[
B=\left[
\begin{array}
[c]{cc}%
A & b\\
c & d
\end{array}
\right]  \in K_{n}
\]
with $A$ an $(n-1)\times(n-1)$ matrix then%
\[
\det(A)=d\det B
\]
so%
\[
\left\vert \det A\right\vert =\left\vert d\right\vert =\sqrt{1-\left\Vert
b\right\Vert ^{2}}
\]
If $d\neq0$ then%
\[
cA^{-1}=-\frac{b^{\ast}}{\bar{d}}.
\]

\end{lemma}

\begin{proof}
By definition of $K_{n},$ $BB^{\ast}=I$. Multiplying out we have%
\[
AA^{\ast}+cc^{\ast}=I,Ac^{\ast}+b\bar{d}=0,cc^{\ast}+|d|^{2}=1
\]
so%
\[
\left[
\begin{array}
[c]{cc}%
A & b\\
c & d
\end{array}
\right]  \left[
\begin{array}
[c]{cc}%
I & c^{\ast}\\
0 & \bar{d}%
\end{array}
\right]  =\left[
\begin{array}
[c]{cc}%
A & 0\\
c & cc^{\ast}+\left\vert d\right\vert ^{2}%
\end{array}
\right]  =\left[
\begin{array}
[c]{cc}%
A & 0\\
c & 1
\end{array}
\right]
\]
thus%
\[
\bar{d}\det B=\det A.
\]
The first assertion of the lemma follows since $\left\Vert b\right\Vert
^{2}+\left\vert d\right\vert ^{2}=1.$

To prove the second part $B^{\ast}B=I$ implies that
\[
A^{\ast}b+dc^{\ast}=0.
\]
Taking adjoints%
\[
b^{\ast}A=-\bar{d}c
\]
so if $d\neq0$
\[
cA^{-1}=-\frac{b^{\ast}}{\bar{d}}.
\]

\end{proof}

\subsection{An elementary estimate}

Let for $u\in F^{m}$ and $0<\alpha<1$
\[
\psi_{\alpha,u,m}(x)=(1+\sum_{i=1}^{m}|x_{i}|^{2}+|\left\langle
u,x\right\rangle |^{2})\prod_{i=1}^{m}(1+|x_{i}|^{2})
\]
then

\begin{lemma}
\label{mainLemma}If $0\leq\varepsilon<1$ then there exists $0<C_{\alpha
,m,\varepsilon}$ $<\infty$ such that%
\[
\int_{F^{m}}\psi_{\alpha,u,m}(x)^{-\frac{\dim_{\mathbb{R}}F}{2}}dx\leq
C_{\alpha,m,\varepsilon}(1+\left\Vert u\right\Vert ^{2})^{-\alpha
\varepsilon\frac{\dim_{\mathbb{R}}F}{2}}.
\]

\end{lemma}

If $u=0$ then since%
\[
\psi_{\alpha,0,m}(x)=(1+\sum_{i=1}^{m}|x_{i}|^{2})\prod_{i=1}^{m}%
(1+|x_{i}|^{2})\geq\prod_{i=1}^{m}(1+|x_{i}|^{2})^{1+\frac{1}{m}}
\]
we see that
\[
I_{\alpha,m}=\int_{F^{m}}\psi_{\alpha,u,m}(x)^{-\frac{\dim_{\mathbb{R}}F}{2}%
}dx<\infty\text{.}
\]
If $\left\Vert u\right\Vert \leq1$ then
\[
(1+\left\Vert u\right\Vert ^{2})^{-\alpha\varepsilon\frac{\dim_{\mathbb{R}}%
F}{2}}\geq2^{-\alpha\varepsilon\frac{\dim_{\mathbb{R}}F}{2}}.
\]
Also,%
\[
\psi_{\alpha,u,m}(x)\geq\psi_{\alpha,0,m}(x)
\]
so%
\[
\int_{F^{m}}\psi_{\alpha,u,m}(x)^{-\frac{\dim_{\mathbb{R}}F}{2}}%
dx\leq2^{-\alpha\varepsilon\frac{\dim_{\mathbb{R}}F}{2}}I_{\alpha
,m}(1+\left\Vert u\right\Vert ^{2})^{-\alpha\varepsilon\frac{\dim_{\mathbb{R}%
}F}{2}}.
\]
Thus we need only prove the result for $\left\Vert u\right\Vert \geq1$. Set
$v=\frac{u}{\left\Vert u\right\Vert }$. Let $W_{u,1}=\{x\in\mathbb{R}%
^{m}||(v,x)|\leq1\}$ and $W_{u,2}=\{x\in\mathbb{R}^{m}||(v,x)|\geq1\}$. Then%
\[
\int_{F^{m}}\psi_{\alpha,u,m}(x)^{-\frac{\dim_{\mathbb{R}}F}{2}}%
dx=\int_{W_{u,1}}\psi_{\alpha,u,m}(x)^{-\frac{\dim_{\mathbb{R}}F}{2}}%
dx+\int_{W_{u,2}}\psi_{\alpha,u,m}(x)^{-\frac{\dim_{\mathbb{R}}F}{2}}dx.
\]
We have on $W_{u,2}$%
\[
\psi_{\alpha,u,m}(x)\geq(1+\sum_{i=1}^{m}|x_{i}|^{2}+\left\Vert u\right\Vert
^{2})^{\alpha}\prod_{i=1}^{m}(1+|x_{i}|^{2})\geq
\]%
\[
(1+\left\Vert u\right\Vert ^{2})^{\varepsilon}\psi_{(1-\varepsilon)\alpha
,,m}(x).
\]
So%
\[
\int_{W_{u,2}}\psi_{\alpha,u,m}(x)^{-\frac{\dim_{\mathbb{R}}F}{2}}dx\leq
I_{(1-\varepsilon)\alpha,m}(1+\left\Vert u\right\Vert ^{2})^{-\alpha
\varepsilon\frac{\dim_{\mathbb{R}}F}{2}}.
\]
We now consider the integral over $W_{u,1}$. Note that
\[
\int_{F^{m}}\psi_{\alpha,u,m}(x)^{-\frac{\dim_{\mathbb{R}}F}{2}}dx=\int%
_{F^{m}}\psi_{\alpha,T_{\xi}(u),m}(x)^{-\frac{\dim_{\mathbb{R}}F}{2}}dx
\]
with $T_{\xi}(u)=(\xi_{1}u_{1},...,\xi_{m}u_{m})$ and $\left\vert \xi
_{i}\right\vert =1$. Thus we may assume that $u_{i}\in\mathbb{R}$ and
$u_{i}\geq0$ in the proof of the lemma. Let $q$ be such that $u_{q}$ \ is
maximum. Define new linear coordinates by%
\[
y_{i}=x_{i},i<q,y_{q}=\left\langle x,v\right\rangle ,y_{i}=x_{i},q<i\leq m.
\]
Then the Jacobian of the transformation $x\mapsto y$ is $v_{q}^{\dim
_{\mathbb{R}}F}$. We also note that since $\left\Vert v\right\Vert =1,1\geq$
$v_{q}\geq\frac{1}{\sqrt{m}}.$ Also,%
\[
x_{q}=\frac{y_{q}-\sum_{i\neq q}v_{i}y_{i}}{v_{q}}
\]
and%
\[
\left\langle x,u\right\rangle =\left\Vert u\right\Vert y_{q}\text{.}
\]
Writing out $\psi_{\alpha,u,m}(x)$ in terms of the $y_{i}$ we have
($0<\varepsilon<1)$
\[
\prod_{i<q}(1+\left\vert y_{i}\right\vert ^{2})\left(  1+\left\vert
\frac{y_{q}-\sum_{i\neq q}v_{i}y_{i}}{v_{q}}\right\vert ^{2}\right)  \times
\]%
\[
\prod_{i>q}(1+\left\vert y_{i}\right\vert ^{2})(1+\sum_{j\neq q}\left\vert
y_{j}\right\vert ^{2}+\left\vert \frac{y_{q}-\sum_{i\neq q}v_{i}y_{i}}{v_{q}%
}\right\vert ^{2}+\left\Vert u\right\Vert ^{2}|y_{q}|^{2})^{\alpha}\geq
\]%
\[
\prod_{i\neq q}(1+\sum_{j\leq i}\left\vert y_{j}\right\vert ^{2})\left(
1+\sum_{.j\neq q}\left\vert y_{j}\right\vert ^{2}+\left\Vert u\right\Vert
^{2}|y_{q}|^{2}\right)  ^{\alpha}\geq
\]%
\[
\prod_{i\neq q}(1+\sum_{j\leq i}\left\vert y_{j}\right\vert ^{2})\left(
1+\sum_{.j\neq q}\left\vert y_{j}\right\vert ^{2}\right)  ^{(1-\varepsilon
)\alpha}\left(  1+\left\Vert u\right\Vert ^{2}|y_{q}|^{2}\right)
^{\varepsilon\alpha}
\]%
\[
=\psi_{(1-\varepsilon),\alpha,0,m-1}(y)\left(  1+\left\Vert u\right\Vert
^{2}|y_{q}|^{2}\right)  ^{\varepsilon\alpha}.
\]
Thus (since the real Jacobian of the transformation $x\rightarrow y$ is
$\frac{1}{y_{q}^{\dim_{\mathbb{R}}F}}$ and $y_{q}\geq\frac{1}{^{\sqrt{m}}}$
\[
\int_{W_{u,1}}\psi_{\alpha,u,m}(x)^{-\frac{\dim_{\mathbb{R}}F}{2}}dx\leq
\frac{1}{m^{\frac{\dim_{\mathbb{R}}F}{2}}}I_{(1-\varepsilon)\alpha,m-1}%
\int_{\left\vert z\right\vert \leq1}\frac{dz}{\left(  1+\left\Vert
u\right\Vert ^{2}|z|^{2}\right)  ^{\varepsilon\alpha\frac{\dim_{\mathbb{R}}%
F}{2}}}.
\]
Thus, to complete the argument in the case at hand we need to show that if
$a\in\mathbb{R}$, $a\geq1,0<\beta<1$ then
\[
\int_{\left\vert y\right\vert \leq1}\left(  1+a^{2}|y|^{2}\right)  ^{-\beta
/2}dy\leq C_{\beta}(1+a^{2})^{-\beta}.
\]
$F=\mathbb{R}$: then
\[
\int_{0}^{1}\frac{dx}{(1+a^{2}x^{2})^{\frac{\beta}{2}}}\leq a^{-\beta}\int%
_{0}^{1}\frac{1}{x^{\beta}}dx=\frac{1}{1-\beta}a^{-\beta}.
\]
Since $a\geq1$ we have $2a^{2}\geq1+a^{2}$ so $2^{-\beta}a^{-\beta}%
\leq(1+a^{2})^{-\beta/2}$.Take%
\[
C_{\beta}=\frac{1}{1-\beta}2^{\beta}.
\]
$F=\mathbb{C}:$
\[
\int_{\left\vert x\leq1\right\vert }\frac{dx}{(1+a^{2}\left\vert x\right\vert
^{2})^{\beta/2}}=2\pi\int_{0}^{1}\frac{sds}{(1+a^{2}s^{2})^{\beta/2}}
\]%
\[
=\pi\int_{0}^{1}\frac{ds}{(1+a^{2}s)^{\beta/2}}\leq\frac{\pi}{a^{\beta}}%
\int_{0}^{1}\frac{ds}{s^{\beta/2}}\leq\frac{\pi}{a^{\beta}(1-\beta/2)}
\]
Take%
\[
C_{\beta}=\frac{\pi}{1-\beta/2}2^{\beta}.
\]

\end{document}